\documentclass[12pt]{amsart}
\usepackage{amsmath,amssymb}

\usepackage{cite}

\theoremstyle{plain}
\newtheorem{theorem}{Theorem}
\newtheorem{corollary}{Corollary}
\newtheorem{proposition}{Proposition}

\theoremstyle{definition}
\newtheorem{condition}{Condition}

\theoremstyle{remark}
\newtheorem{remark}{Remark}

\def\lra{\longrightarrow}

\def\a{\alpha}
\def\ka{\varkappa}
\def\cN{{\mathcal{N}}}
\def\cE{{\mathcal{E}}}
\def\ga{\gamma}
\def\dt{\delta}
\def\e{\varepsilon}
\def\si{\sigma}
\def\ph{\varphi}

\newcommand{\BL}{\biggl}
\newcommand{\BR}{\biggr}

\def\wt{\widetilde}
\let\rom\textup

\begin{document}

\author{M.~G.~Shur}

\title[Strong ratio limit theorems associated with random walks]%
{Strong ratio limit theorems associated\\ with random walks on
unimodular groups}

\address{Moscow Institute of Electronics and Mathematics,\newline
\indent National Research University Higher School of Economics}

\date{\today}



\begin{abstract}
Strong ratio limit theorems associated with a broad class of spread
out random walks on unimodular groups were proved in the preceding
paper~\cite{9}, where these random walks were assumed to have the
convergence parameter $R=1$. In the present paper, we study the
case of an arbitrary $R\ge1$ and clarify the role of the condition
that the group is unimodular.
\end{abstract}

\maketitle

\section{Introduction}\label{s1}

The final section of the paper~\cite{9} contains strong ratio limit
theorems (SRLT) associated with a broad family of random walks on
unimodular groups. The conditions imposed in~\cite{9} on these
random walks in particular guarantee that they are irreducible in
the sense of~\cite{3} and the convergence parameter $R$ of each of
these random walks satisfies $R=1$. (Following~\cite{9}, we mainly
borrow the probability part of our terminology from~\cite{3,6} and
the algebraic part, from~\cite{7}.) The aim of the present paper is
to transfer these results in~\cite{9} to the case of arbitrary
$R\ge1$.

In this connection, recall that the SRLT suggested in~\cite{9} are
slightly different from the traditional ones (e.g.,
cf.~\cite{1,2,3,4}). As long as some SRLT of the traditional kind
can be applied to a random walk $X=(X_n;n\ge0)$ on a locally
compact group $E$, this SRLT can be used to obtain information
about the existence and values of limits of the form
\begin{equation}\label{eq1}
    \lim_{n\to\infty}\frac{\ka(P^{n+m}f)}{\mu(P^ng)},\qquad m\ge0,
\end{equation}
where $P$ is the transition operator of the random walk and the
Borel functions $f,g\ge0$ defined on $E$, together with the
probability measures $\ka$ and $\mu$ defined on the set of all
Borel subsets of $E$, satisfy appropriate
conditions~\cite{1,2,3,4}. (Here and in what follows, we write,
say, $\ka(f)=\int f\,d\ka$ if the integral exists.) However, many
authors only consider the case in which these measures are
concentrated at some point $x\in E$ and hence the ratio occurring
in~\eqref{eq1} coincides with $P^{n+m}f(x)/P^ng(x)$ (e.g.,
see~\cite{1}).

Of the conditions included in the traditional SRLT, those implying
the existence of limits as $n\to\infty$ of ratios of the form
$\ka(P^{n+1}f)/\ka(P^nf)$ for one or maybe all admissible functions
$f$ play a special role (see \cite{1,2,3,4}). The verification of
such conditions (apart from the case of symmetric random
walks~\cite{8}) often encounters serious difficulties. To avoid
such conditions as much as possible, we considered SRLT of a~more
general type in~\cite{9}. In these theorems, limits of the
form~\eqref{eq1} are replaced by some analogs in which the integer
parameter $n$ tends to infinity avoiding some set $\cN$ of density
$0$, that is, a set $\cN$ of positive integers such that
\begin{equation}\label{eq2}
    \lim_{n\to\infty}\BL[\frac1n q_n(\cN)\BR]=0,
\end{equation}
where $q_n(\cN)$, $n\ge0$, is the number of elements $m\in\cN$ that
do not exceed~$n$. These SRLT were referred to as SRLT with
exceptional parameter sets in~\cite{9}.

The main results of the present paper (see Theorems~\ref{th1}
and~\ref{th2}), which deal with spread out irreducible random walks
on unimodular groups, are SRLT of this kind as well.
Following~\cite{9,10}, we assume that for each of the random walks
in question there exists a unique (up to a factor) $R$-invariant
measure taking finite values on compact sets, where $R$ has the
same meaning as above, but now the case of $R\ne1$ is of main
interest. The proofs of Theorems~\ref{th1} and~\ref{th2} use the
results in~\cite{9,10}. As to the assumption that the groups
considered are unimodular, it was also introduced in~\cite{9}, but
in the present paper we establish that it is necessary in our
theory (see Theorem~\ref{th3}).

In conclusion of the introduction, let us give some notation and
definitions. In what follows, $E$ is a locally compact group with
countable base, $\cE$ is the Borel $\si$-algebra on $E$, and $\pi$
is the right Haar measure on $\cE$; we also set
$\cE_+=\{A\in\cE\colon\pi(A)>0\}$. The group operation on $E$ is
interpreted as multiplication.

On $E$, we specify a random walk $X=(X_n;n\ge0)$ with a law $v$,
where $v$ is a probability measure on $\cE$, and a transition
operator~$P$. (Thus, $X$ is a homogeneous Markov chain on the
measurable space $(E,\cE)$ with transition probabilities
$p(x,A)=v(x^{-1}A)$, $x\in E$, $A\in\cE$.) If, for some positive
integer $n$, the convolution power $v^n$ is nonsingular with
respect to $\pi$, then $X$ is called a \textit{spread out random
walk}. As is shown by \cite[Proposition~1]{10}, such a random walk
is irreducible in the traditional sense~\cite{5} if and only if it
is irreducible with respect to the measure~$\pi$ in the sense of
the theory of irreducible Markov chains~\cite{3}. In what follows,
unless specified otherwise, the irreducibility of random walks is
understood in the traditional sense.

\section{Auxiliary assertions}\label{s2}

The main aim of this section is to prove the following assertion,
which, in particular, contains a full description of small
functions and measures~\cite{3} for a broad class of random walks.

\begin{proposition}\label{p1}
Let a random walk $X$ be spread out, irreducible, and aperiodic in
the sense of~\rom{\cite{3}}. Let $S_0$ be the family of all bounded
Borel functions $f\lra[0,\infty)$ such that the set $\{x\in E\colon
f(x)>0\}$ belongs to $\cE_+$ and is relatively compact. Then
\begin{enumerate}
    \item[\rom{(a)}] Each function $f\in S_0$ and each measure of
    the form $\mu=f\pi$ with $f\in S_0$ are small for $X$. \rom(By
    definition, $f\pi(A)=\int_Af\,d\pi$, $A\in\cE$.\rom)
    \item[\rom{(b)}] If the law $v$ of $X$ is compactly supported,
    then, conversely, all functions and measures small for $X$ have
    the form described in~\rom{(a)}.
\end{enumerate}
\end{proposition}

\begin{proof}
By \cite[Proposition~3]{10}, there exists an open set $U\subset E$
such that the set $sU=\{y\in E\colon y=sx,x\in U\}$ and the measure
$1_{sU}\pi$ are small for $X$. (Here the symbol $1_A$ stands for
the indicator function of a set $A\in\cE$.) Given $f\in S_0$, take
a compact set $F\subset E$ such that $f=0$ outside~$F$. Since $E$
is locally compact, it follows that this compact set can be covered
by finitely many sets of the form $sU$, $s\in F$. In view of the
aperiodicity of $X$ and \cite[Corollary~2.1]{3}, we can conclude
that $f$ is a small function, thus proving the part of~(a)
pertaining to small functions. A similar argument gives the
remaining part of~(a).

Now let us proceed to~(b). Fix a small function~$f$. If $g\in S_0$,
then $aP^mg\ge f$ for some $m\ge1$ and $a>0$. Since the measure
$v^m$ is now compactly supported, it follows that $P^mg=0$ outside
some compact set and hence $f=0$ outside the same set. This
justifies the part of~(b) pertaining to small functions. The case
of small measures can be analyzed in a similar way.
\end{proof}

\begin{corollary}\label{cor1}
If the conditions indicated in the first sentence of
Proposition~\rom{\ref{p1}} are satisfied, then every compact set
contained in $\cE_+$ is a small set for~$X$.
\end{corollary}

Now recall that a measure $\mu$ on $\cE$ is said (see~\cite{10} and
many other papers) to be $r$-invariant for $X$, $0<r<\infty$, if
$\mu(E)>0$ and $\mu=r\mu P$, where the measure $\mu P$ is defined
by the formula
\begin{equation*}
    \mu P(A)=\int\rho(x,A)\mu(dx),\qquad A\in\cE.
\end{equation*}

\begin{corollary}\label{cor2}
Under the same assumptions as in Corollary~\rom{\ref{cor1}}, every
$r$-invariant measure $\mu$, $0<r<\infty$, taking finite values on
compact sets is also $r$-invariant in the sense of~\rom{\cite{3}}
\rom(that is, $\mu(A)<\infty$ for at least one and hence for all
small $A\subset E$\rom), and moreover, $\mu(A)>0$ for the same $A$
and for all open $A\subset E$.
\end{corollary}

\begin{proof}
Every compact set $A_1\in\cE$ is a small set by
Corollary~\ref{cor1} and $\mu(A_1)<\infty$ by the assumptions of
the corollary; it follows that $0<\mu(A)<\infty$ for any small $A$
\cite[Proposition~5.6]{3}. If a set $B\subset E$ is open, then
$\pi(B)>0$; hence $B$ contains some compact set $A_1\in\cE_+$, and
we conclude that $\mu(B)>0$.
\end{proof}

\section{Strong ratio limit theorems}
\label{s3}

Consider the set $C_0(E)$ of all continuous compactly supported
functions on~$E$. We introduce the following condition.

\begin{condition}\label{A}
For every nonnegative function $f\in C_0(E)$, there exists a $\ga>0$
and a $j\ge1$ such that $v^j\ge\ga(f\pi)$.
\end{condition}

This condition coincides with condition~(13A) in~\cite{9} and
implies, by \cite[Lemma~13.1]{9}, that the random walk $X$ is spread
out, irreducible, and aperiodic. Condition~\ref{A} is necessarily
satisfied if $X$ is spread out and irreducible and has a law $v$
whose support is not contained in any coset of any proper normal
subgroup of $E$ \cite[Chap.~3]{6}.

We also need the following condition.

\begin{condition}\label{B}
There exists a unique (up to a positive factor) $R$-invariant
measure for $X$ taking finite values on compact sets, where $R$ is
again the convergence parameter of~$X$.
\end{condition}

By using \cite[Theorem~5.8 and Proposition~5.6]{3} and
Corollary~\ref{cor1}, we see that Condition~\ref{B} applies, say,
if the random walk in question is spread out, irreducible, and
$R$-recurrent. (The last condition means that $X$ is an
$R$-recurrent Markov chain~\cite{3}.) It is worth noting that, as
\cite[Proposition~5.6]{3} and our Corollary~\ref{cor2} show, the
measure $\nu$ is necessarily absolutely continuous with respect to
$\pi$ under Conditions~\ref{A} and~\ref{B}.

Next, recall that a Borel function $\ph\colon E\lra(0,\infty)$ is
called an \textit{exponential} on $E$ if $\ph(xy)=\ph(x)\ph(y)$ for
any $x,y\in E$ and that Condition~\ref{B} guarantees by
\cite[Theorem~1]{10} that there exists a unique continuous
exponential $\ph$ on $E$ such that
\begin{equation}\label{eq3}
    \int\ph\,dv=\frac1R.
\end{equation}
In what follows, the symbol $\ph$ is used to denote this
exponential.

\begin{theorem}\label{th1}
Let the group $E$ be unimodular, and let the random walk $X$
satisfy Conditions~\rom{\ref{A}} and~\rom{\ref{B}} and have
compactly supported law~$v$. Then there exists a subsequence $\cN$
of positive integers of density $0$ \rom(see~\eqref{eq2}\rom) such
that
\begin{equation}\label{eq4}
    \lim_{\substack{n\to\infty\\n\notin\cN}}
    \frac{P^nf(x)}{P^ng(y)}
    =\frac{\pi(\ph^{-1}f)\ph(x)}{\pi(\ph^{-1}g)\ph(y)}
    =\frac{\nu(f)\ph(x)}{\nu(g)\ph(y)}
\end{equation}
for any $x,y\in E$ and $f,g\in C_0(E)$ except for $g\in C_0(E)$
with $\nu(g)=0$, where $\nu$ is the measure indicated in
Condition~\rom{\ref{B}}. Moreover,
\begin{equation}\label{eq5}
    \lim_{\substack{n\to\infty\\n\notin\cN}}
    \frac{P^{n+m}g(y)}{P^ng(y)}=\frac1{R^m},\qquad y\in E,
\end{equation}
for any positive integer $m$, and the convergence to the limits
in~\eqref{eq4} and~\eqref{eq5} is uniform with respect to $x$ and
$y$ ranging over an arbitrary compact set in~$E$.
\end{theorem}
\begin{proof}
Without loss of generality, we assume in what follows that
$\nu=\ph^{-1}\pi$ \cite[Theorem~1]{10}. Consider the spread out
random walk $\wt X$ on $E$ with the law $\wt v=R\ph v$. (The
relation $\wt v(E)=1$ follows from~\eqref{eq3}.) Just as
in~\cite[Sec.~4]{10}, we find that the transition operator $\wt P$
of this random walk is related to $P$ by
\begin{equation}\label{eq6}
    \wt P^nf=R\ph^{-1}P^n(f\ph),\qquad n\ge1,
\end{equation}
where $f$ is an arbitrary bounded Borel function on~$E$. Thus, just
as in~\cite{10}, the operator $\wt P$ is obtained from $P$ by a
well-known similarity transformation~\cite{3}, whence the random
walk proves to be irreducible and have the convergence parameter
$\wt R=1$.

Just as for an arbitrary random walk on $E$, the measure $\pi$ is
invariant for $\wt X$. Moreover, it is the unique (up to a factor)
measure that is invariant for $\wt X$, is absolutely continuous
with respect to $\pi$, and takes finite values on compact sets.
Indeed, assume that some measure $\nu_1$ has the last three
properties but does not coincide with any of the measures $\a\pi$,
$\a>0$. The one can readily establish the invariance of the measure
$\ph^{-1}\nu_1$ for $X$ with the use of~\eqref{eq3}. However, this
measure differs from any of the measures $\a\ph^{-1}\pi=\a\nu$,
$\a>0$, and by Condition~\ref{B} this is only possible if
$\nu_1(A)=\infty$ for some compact set $A\subset E$. Thus, we
arrive at a contradiction with the choice of $\nu_1$, and hence we
have established the desired uniqueness of the measure~$\pi$.

Next, by Corollary~\ref{cor2}, every measure invariant for $\wt X$
is dense in $E$, that is, takes nonzero values on open sets in $E$.
This, together with the preceding discussion, permits us to use
Theorem~13.1 in~\cite{9} and find a subsequence $\cN$ of density
$0$ of positive integers such that
\begin{equation}\label{eq7}
    \lim_{\substack{n\to\infty\\n\notin\cN}}
    \frac{\wt P^nf(x)}{\wt P^ng(y)}
    =\frac{\pi(f)}{\pi(g)}
\end{equation}
provided that the functions $f$ and $g$ are nonnegative (which is
from now on assumed until we drop this condition) and chosen in
accordance with the theorem to be proved (which is assumed until
the end of the proof). As to the character of the convergence to
the limit in~\eqref{eq7}, it is as indicated in the theorem.

Consider the special case of $f=\wt P^mg$, $m\ge1$, where the
continuity of $\wt P^mg$ is guaranteed by that of $g$ and the
Feller property of random walks \cite[Exercise 5.14]{6}, while the
compactness of its support follows from that of the support of the
measure $v^m$. Relation~\eqref{eq7} with $x=y$ can be written as
\begin{equation}\label{eq8}
    \lim_{\substack{n\to\infty\\n\notin\cN}}
    \frac{\wt P^{n+m}g(y)}{\wt P^ng(y)}=1,
\end{equation}
because the measure $\pi$ is invariant for~$\wt X$.

Now we return to~\eqref{eq7} and use~\eqref{eq6} to derive the
desired relations~\eqref{eq4} with $\ph f$ and $\ph g$ instead of
$f$ and $g$, respectively, from~\eqref{eq8}. Thus, to
obtain~\eqref{eq4} in the desired form, we should apply the version
of~\eqref{eq4} just obtained to $\ph^{-1}f$ and $\ph^{-1}g$,
respectively, instead of~$f$ and~$g$. In a similar way, we
derive~\eqref{eq5} from~\eqref{eq8}.

It remains to get rid of the requirement that $f$ and $g$ are
nonnegative. Assuming that this requirement is violated, we
represent $f$ and $g$ as $f=f_1-f_2$ and $g=g_1-g_2$, where
$f_1(x)=\max\{0,f(x)\}$, $f_2(x)=-\min\{0,f(x)\}$, and $g_1$ and
$g_2$ are defined in a similar way. If $\nu(g_1)>0$, then, by using
what was just proved, we successively compute the limits of the
ratios
\begin{equation*}
    \frac{P^nf_i(x)}{P^ng_1(y)},\quad i=1,2,\quad
    \frac{P^nf(x)}{P^ng_1(y)},\qquad
    \frac{P^ng_1(x)}{P^ng(y)}
\end{equation*}
as $n\to\infty$ ($n\notin\cN$) and then the limit indicated
in~\eqref{eq4}, after which relations~\eqref{eq4} become obvious.
The argument for~\eqref{eq4} with $\nu(g_2)>0$ as well as
for~\eqref{eq5} is equally simple. The proof of the theorem is
complete.
\end{proof}

\begin{remark}\label{rk1}
In the framework of traditional SRLT, one can suggest a natural
analog of relations~\eqref{eq4} (see Eq.~\eqref{eq10} below).
Indeed, under appropriate assumptions including those imposed in
Theorem~1, the paper~\cite{1} permits one to apply the relation
\begin{equation}\label{eq9}
    \lim_{n\to\infty}\frac{P^nf(x)}{P^ng(x)}=\frac{\nu(f)}{\nu(g)}
\end{equation}
for $x\in E$. (We retain the same notation as in
Theorem~\ref{th1}.) If we take some $x,y\in E$ and set $z=yx^{-1}$
and $g_z(u)=g(zu)$, $u\in E$, then simple calculations give
$P^ng_z(x)=P^ng(y)$ and
$\nu(g_z)=\ph(z)\nu(g)=\ph(y)\ph(x^{-1})\nu(g)$. Hence
Eq.~\eqref{eq9} with $g$ replaced by $g_z$ shows that
\begin{equation}\label{eq10}
    \lim_{n\to\infty}\frac{P^nf(x)}{P^ng(y)}
        =\frac{\ph(x)\nu(f)}{\ph(y)\nu(g)}.
\end{equation}
Thus, we have obtained the desired analog of Eqs.~\eqref{eq4}.
\end{remark}

The following assertion again uses the measure~$\nu$ occurring in
Condition~\ref{B}, but the sequence $\cN$ may differ from the one
considered in Theorem~\ref{th1}.

\begin{theorem}\label{th2}
Let the assumptions of Theorem~\ref{th1} be satisfied. Then there
exists a sequence $\cN$ of positive integers of density~$0$ such
that
\begin{align}\label{eq11}
    \lim_{\substack{n\to\infty\\n\notin\cN}}
           \frac{\ka(P^nf)}{\mu(P^ng)}
          &=\frac{\ka(\ph)\pi(\ph^{-1}f)}{\mu(\ph)\pi(\ph^{-1}g)}
          =\frac{\ka(\ph)\nu(f)}{\mu(\ph)\nu(g)},\\ \label{eq12}
    \lim_{\substack{n\to\infty\\n\notin\cN}}
           \frac{\ka(P^{n+m}g)}{\mu(P^ng)}
          &=\frac1{R^m},\qquad m\ge1,
\end{align}
where $f$ and $g$ are bounded Borel functions defined on~$E$ and
vanishing outside some compacts set, $\nu(g)\ne0$, and the
probability measures $\mu$ and $\ka$ defined on $\cE$ have the form
indicated in Proposition~\rom{\ref{p1}\,(a)}.
\end{theorem}

\begin{proof}
Assume for now that the functions occurring in the theorem are
continuous. By applying~\eqref{eq4} and the standard rule of
passing to the limit in the integrand, we find that
\begin{equation*}
    \frac{\ka(P^nf)}{P^ng}\lra\frac{\nu(f)\ka(\ph)}{\nu(g)\ph}
\end{equation*}
everywhere in $E$ as $n\to\infty$ ($n\notin\cN$), where $\cN$ is so
far chosen in the same way as in Theorem~\ref{th1}. From this, we
derive~\eqref{eq11} by a similar argument provided that
$\nu(f)\ne0$. If, on the other hand, $\nu(f)=0$, then, in view of
what has already been proved, we can use~\eqref{eq11} with the
functions $f_m=f+\frac1mg$, $m\ge1$, instead of $f$ and then pass
to the limits as $m\to\infty$ to obtain~\eqref{eq11} again with the
original function~$f$. Thus, \eqref{eq11} applies to any functions
$f$ and $g$ of the form we have just spoken of.

Now we fix Borel functions $f,g\ge0$ and measures $\mu$ and $\ka$
satisfying the assumptions of the theorem and find continuous
functions $f_n,g_n\ge0$ vanishing outside a common compact set and
satisfying $\pi(|f-f_n|)\to0$ and $\pi(|g-g_n|)\to0$ as
$n\to\infty$. It follows from the first relation
in~\cite[Eq.~(12.8)]{9} that
\begin{equation*}
    \liminf \frac{\ka(P^nf)}{\mu(P^ng)}
       \ge\frac{1-\e}{1+\e}\liminf\frac{\ka(P^nf_m)}{\mu(P^ng_m)}
\end{equation*}
for any given $\e\in(0,1)$ and all sufficiently large $m\ge1$.
(Here and in the following, we assume that a subsequence $\cN$ of
positive integers of density~$0$ is chosen in an appropriate
way~\cite{9} and that the lower limits correspond to $n\to\infty$,
$n\notin\cN$.) According to what was said at the beginning of the
proof, the second of these lower limits is equal to the ratio
\begin{equation*}
    \frac{\ka(\ph)\nu(f_m)}{\mu(\ph)\nu(g_m)},
\end{equation*}
so that if we first let $m\to\infty$ and then let $\e\to0$, then we
conclude from the last inequality that
\begin{equation}\label{eq13}
    \liminf \frac{\ka(P^nf)}{\mu(P^ng)}
       \ge\frac{\ka(\ph)\nu(f)}{\mu(\ph)\nu(g)}.
\end{equation}

Just in the same way, for $\nu(f)>0$ the similar lower limit of the
ratio $\mu(P^nf)/\ka(P^nf)$ cannot be less than $1/\dt$, where
$\dt$ is equal to the right-hand side of~\eqref{eq13}. From this
and~\eqref{eq13}, we readily derive~\eqref{eq11} (cf.~the proof of
Theorems~12.4 and~13.2 in~\cite{9}). If $f=0$ $\nu$-a.e., then, to
achieve the same goal, it suffices to compare the results of the
substitution of the functions $f+g$ and $g$ for $f$
into~\eqref{eq11}.

In the special case of $f=P^mg$, \eqref{eq11} implies \eqref{eq12}
(cf.~the proof of~\eqref{eq8}; here it is recommended to take into
account the $R$-invariance of the measure~$\nu$).

Finally, just as in the preceding proof, one can readily drop the
requirement that $f$ and $g$ be nonnegative, thus completing the
proof of Theorem~\ref{th2}.
\end{proof}

\begin{remark}\label{rk2}
If $f,g\in C_0(E)$ and $\nu(g)=0$, then one can apply
relations~\eqref{eq11} and~\eqref{eq12} with arbitrary compactly
supported probability measures $\mu$ and $\ka$ defined on $\cE$. To
prove this, it suffices to note that the first paragraph of the
proof of Theorem~\ref{th2} does not use the possibility of
representation of these measures in the form indicated in the
theorem.
\end{remark}

In conclusion, it remains to explain why one cannot omit the
assumption that the group $E$ is unimodular from Theorems~\ref{th1}
and~\ref{th2}.

\begin{theorem}\label{th3}
If the group $E$ is not unimodular, then there does not exist an
irreducible spread out random walk on $E$ with any convergence
parameter $R\ge1$ possessing a unique \rom(up to a factor\rom)
$R$-invariant measure taking finite values on compact sets.
\end{theorem}

\begin{proof}
We restrict ourselves to the case of $R=1$, because otherwise one
could consider the random walk $\wt X$ with convergence parameter
$\wt R=1$ instead of~$X$ (see the proof of Theorem~\ref{th1}).

Let $X$ be a random walk on $X$ with the corresponding properties
possessing a unique (up to a factor) invariant measure taking
finite values on compact sets. Clearly, this measure coincides with
$\pi$ up to a factor. Next, assume that functions $f,g\ge0$ belong
to the family $C_0(E)$ and $\pi(g)\ne0$. By~\cite[Eq.~(13.13)]{9},
\begin{equation}\label{eq14}
    \lim\frac{P^nf(e)}{P^ng(e)}=\frac{\pi(f)}{\pi(g)},
\end{equation}
where $\cN$ is an appropriate subsequence of positive integers of
density~$0$, $e$ is the identity element of the group, and the
limit in \eqref{eq14}--\eqref{eq16} is taken as $n\to\infty$,
$n\notin\cN$; we point out that, in contrast to the text adjacent
to the proof of~\cite[Eq.~(13.13)]{9}, the derivation
of~\cite[Eq.~(13.13)]{9} itself does not use the assumption that
$E$ is unimodular. Hence, in view of~\cite[Sec.~15.27(c)]{7}, we
obtain
\begin{equation}\label{eq15}
    \lim\frac{P^ng(x)}{P^ng(e)}=\lim\frac{P^ng_x(e)}{P^ng(e)}
        =\frac{\pi(g_x)}{\pi(g)}=\Delta(x)
\end{equation}
for $x\in E$, where $\Delta$ is the modular function of $E$ and the
translate $g_x$ of $g$ is defined by analogy with $g_z$ (see
Remark~\ref{rk1}).

Now take a compact set $F\subset E$. By~\cite[Lemmas~13.3
and~13.4]{9}, the family of functions $\{P^ng/P^ng(e); n\ge n_1\}$
with sufficiently large positive integer $n_1$ is equicontinuous on
$F$. Consequently, the convergence in~\eqref{eq15} is uniform with
respect to $x\in F$. Hence, by taking the support of $v$ for $F$,
we obtain, in view of~\eqref{eq15},
\begin{multline}\label{eq16}
    \lim\frac{P^{n+1}g(x)}{P^{n}g(x)}
    =\lim\int\frac{P^{n}g(xy)}{P^{n}g(x)}v(dy)
    \\
    =\int\Delta(xy)v(dy)=\Delta(x)\int\Delta\,dv
    =r^{-1}\Delta(x),
\end{multline}
because $\Delta$ is a continuous exponential on $E$
\cite[Theorem~15.11]{7}; here $r=[\int\Delta\,dv]^{-1}$.

By~\cite[Lemma~13.2]{9}, the first of the limits in~\eqref{eq16}
cannot be less than~$1$, and hence \eqref{eq16}~implies the
inequality $\Delta(x)\ge r$, $x\in E$. At the same time, we have
$\Delta(x^{-1})=\Delta^{-1}(x)\ge r$, so that $r\le\Delta(x)\le
r^{-1}$, $x\in E$, whence it follows that $\Delta(x)\equiv1$.
(Indeed, if $\Delta(x)\ne1$ for some $x\in E$, then, for
sufficiently large positive integer $n$, either
$\Delta(x^n)=[\Delta(x)]^n<r$ or $\Delta(x^n)>r^{-1}$, but both of
the last inequalities are excluded by our argument.) Thus, we
have established that the group $E$ is unimodular
(cf.~\cite[Remark~15.12]{7}).
\end{proof}

\providecommand{\bysame}{\leavevmode\hbox to3em{\hrulefill}\thinspace}

\end{document}